\documentclass[11pt, notitlepage]{article}
\usepackage{amssymb,amsmath,comment}
\catcode`\@=11 \@addtoreset{equation}{section}
\def\thesection{\arabic{section}}

\def\theequation{\thesection.\arabic{equation}}
\catcode`\@=12
\usepackage{colortbl}
\usepackage{a4wide}
\usepackage{parskip}

\newcommand{\ds} {\displaystyle}
\newcommand{\e}{\varepsilon}

\newcommand{\al} {\alpha}

\newcommand{\de} {\delta}
\newcommand{\ga} {\gamma}

\newcommand{\Om} {\Omega}
\newcommand{\ra} {\rightarrow}

\newcommand{\De} {\Delta}
\newcommand{\la} {\lambda}
\newcommand{\La} {\Lambda}
\newcommand{\noi} {\noindent}
\newcommand{\na} {\nabla}

\newcommand{\mb} {\mathbb}
\newcommand{\mc} {\mathcal}
\newcommand{\lra} {\longrightarrow}
\newcommand{\ld} {\langle}
\newcommand{\rd} {\rangle}

\usepackage[all]{xy}
\catcode`\@=11
\def\theequation{\@arabic{\c@section}.\@arabic{\c@equation}}
\catcode`\@=12

\def\QED{\hfill {$\square$}\goodbreak \medskip}

\newtheorem{Theorem}{Theorem}[section]
\newtheorem{Lemma}[Theorem]{Lemma}
\newtheorem{Proposition}[Theorem]{Proposition}

\newtheorem{Definition}[Theorem]{Definition}

\begin{document}
\vspace{0.01in}

\title
{On the Second Eigenvalue of Combination Between Local and Nonlocal $p$-Laplacian  }

\author{ {\bf Divya Goel\footnote{e-mail: divyagoel2511@gmail.com} \; 
		and \;  K. Sreenadh\footnote{
			e-mail: sreenadh@maths.iitd.ac.in}} \\ Department of Mathematics,\\ Indian Institute of Technology Delhi,\\
	Hauz Khaz, New Delhi-110016, India. }

\date{}

\maketitle

\begin{abstract}
\noi In this paper, we study Mountain Pass Characterization of the second eigenvalue of the operator $-\De_p u -\De_{J,p}u$ and  study  shape optimization problems related to these eigenvalues.
\medskip
  
\noi \textbf{Key words:} nonlocal $p$-Laplacian, Eigenvalue problem, Faber-Krahn inequality, nonlocal Hong-Krahn-Szego inequality. 

\medskip

\noi \textit{2010 Mathematics Subject Classification:   35P30, 47J10, 49Q10.} 

\end{abstract}

\section{Introduction}
Let $\Om $ be an open and  bounded domain in $\mathbb{R}^N$ with $C^{1,\alpha}$ boundary. In this article, we  study the following eigenvalue  problem
\begin{equation*}
(P_\la)\;
\left.\begin{array}{rllll}
\mathcal{L}_{J,p}(u) =\la |u|^{p-2}u \text{ in } \Om, \; u=0 \text{ in } \mathbb{R}^N\setminus \Om,
\end{array}
\right.
\end{equation*}
where the operator $\mathcal{L}_{J,p}(u)$ is defined as 
$\mathcal{L}_{J,p}u:= -\De_p u -\De_{J,p}u $, $\De_p(u):= \text{div}(|\na u|^{p-2}\na u)$ is the usual $p$-Laplacian operator and  the nonlocal $p$-Laplacian is given by 
\begin{align*}
\De_{J,p}u(x):= 2 \ds \int_{\mathbb{R}^N} |u(x)-u(y)|^{p-2}(u(x)-u(y))J(x-y)~ dy, \quad 1< p< \infty .
\end{align*}
 Here the kernel $J:\mathbb{R}^N \ra \mathbb{R}$ is a radially symmetric, nonnegative continuous function with compact support, $J(0)>0$ and $\int_{\mathbb{R}^N} J(x)~ dx =1$. 
 Recently, the study of nonlocal equations fascinate a lot of researchers. In particular, equations involving fractional $p$-Laplacian  operator gain lot of attention. In \cite{pal}, Lindgren and Lindqvist studied the eigenvalues of the following problem 
 \begin{equation}\label{fs16}
 \begin{aligned}
 -2 \int_{\mathbb{R}^N} \frac{|u(x)-u(y)|^{p-2}(u(x)-u(y))}{|x-y|^{N+sp}}~dy  = \la |u(x)|^{p-2}u(x) \text{ in } \Om, \; u=0 \text{ in } \mathbb{R}^N\setminus \Om
 \end{aligned}
 \end{equation}
 Here they studied the eigenvalues, viscosity solutions and  the limit case as $p \ra \infty$. Later in \cite{second}, Brasco and Parini studied the problem \eqref{fs16} in an open bounded, possibly disconnected set $\Om \subset \mathbb{R}^N$ and $1<p<\infty$. In this paper,  authors also discussed about the regularity of the eigenfunctions of the operator fractional  $p$-Laplacian and gave the mountain pass characterization of the second eigenvalue of fractional  $p$-Laplacian. Moreover, authors proved the nonlocal Hong-Krahn-Szego inequality. We cite \cite{bisci, goel1, hardy, hitchhker}  and references therein for the work on  equations involving  fractional  $p$-Laplacian.  For the work on second eigenvalue of $p$-Laplacian we cite \cite{cfg, sree1} and references therein. \\
  On the other hand, nonlocal equations involving nonlocal $p$-Laplacian of zero-order, that is, the following problem
  \begin{align}\label{fs17}
  - \int_{\mathbb{R}^N} |u(x)-u(y)|^{p-2}(u(x)-u(y))J(x-y)~ dy= \la |u|^{p-2}u
  \end{align} 
  has been studied in \cite{an1, an3}. In these papers it has been proved that the Rayleigh quotient corresponding to problem \eqref{fs17} is strictly positive.  We refer \cite{an1, an2, an3} and references therein for the work on  equations involving nonlocal $p$-Laplacian of zero-order.\\
 The inspiring point of our work is the work of Del Pezzo et al. (\cite{rossi1}), where
 authors studied the eigenvalue problem of the operator $\mathcal{L}_{J,p}$ and   proved the existence of the eigenfunction of the smallest eigenvalue. In particular, authors proved  the following result:
\begin{Theorem} \label{fsthm4}
 Assume $p\geq 2$.	There exists a sequence of eigenvalues $\{\la_k\}_{k \in \mathbb{N}}$ of the operator $\mathcal{L}_{J,p}$ such that $\la_k\ra +\infty$. The first eigenvalue $\la_1(\Om)$ is simple, isolated and its corresponding eigenfunctions have a constant sign. Moreover, $\la_1(\Om)$ can be characterized by 
	\begin{align*}
	\la_1(\Om):= \inf_{u \in W_0^{1,p}(\Om)}\bigg\{   \int_{\Om}|\na u|^p~dx+ \ds \int_{\mathbb{R}^N}\int_{\mathbb{R}^N} |u(x)-u(y)|^{p}J(x-y)~ dxdy :\int_{\Om}|u|^p~dx=1 \bigg\}.
	\end{align*}
	 Furthermore, every eigenfunction  belongs to $C^{1,\al}(\overline{\Om})$ for some $\al \in (0,1)$. 
	\end{Theorem}
We remark that by using the discrete picone identity as in  \cite{hardy}, one can get $\la_1(\Om)$ is  simple, isolated and eigenfunctions corresponding to eigenvalue other than $\la_1(\Om)$ changes sign for all $1<p<\infty$. The variational characterization of second eigenvalue and the Sharp lower bounds on the first and second eigenvalue  remains open question. In the present paper, we prove the variational characterization of the second eigenvalue of the operator associated to the problem $(P_\la)$. Also, we consider the following shape optimization problems 
\begin{align}\label{fs18}
& \inf\{\la_1(\Om): |\Om|=c  \},\\ \label{fs19}
& \inf  \{\la_2(\Om): |\Om|=c  \},
\end{align}
where $c$ is a positive number. For the optimization problem \eqref{fs18}, we prove the Faber-Krahn inequality (See Theorem \ref{fsthm2}) which says that 
\begin{center}
	``In the class of all domains with fixed volume, the ball has the smallest first eigenvalue."
\end{center}
Corresponding to the optimization problem \eqref{fs19}, we first prove a result for nodal domains (See Lemma \ref{fslem7}) whose statement can be rephrased as
\begin{center}
	``Restriction of an eigenfunction to a nodal domain is not an eigenfunction of this nodal domain." 
\end{center}
This Lemma is due to the nonlocal nature of the operator. Next we prove the Nonlocal Hong-Krahn-Szego inequality for the operator associated to problem $(P_\la)$ (See Theorem \ref{fsthm3}) which states that 
\begin{center}
	``In the class of all domains with fixed volume, the   smallest second eigenvalue is obtained for 	the disjoint union of two balls." 
\end{center}
It implies shape optimization problem  \eqref{fs19} does not admit a solution.  Since the Rayleigh quotient corresponding to problem $(P_\la)$  does not follow the scale invariance, there is significant amount of difference in handling the combined effects  of  $p$-Laplacian and nonlocal $p$-Laplacian of zero order.  With this introduction we will state our main results:
\begin{Theorem}\label{fsthm1}
	Let $1<p<\infty$ and $\Om \subset \mathbb{R}^N$ be an open and bounded set. Then there exists a positive number $\la_2(\Om)$ with the following properties:
	\begin{enumerate}
		\item $\la_2(\Om)$ is an eigenvalue of the operator $\mathcal{L}_{J,p}$.
		\item $\la_2(\Om)> \la_1(\Om)$.
		\item   if $\la > \la_1(\Om)$ is an eigenvalue then $\la\geq \la_2(\Om)$.
	\end{enumerate}
	Furthermore, $\la_2(\Om)$ has the following variational characterization
	\begin{align*} 
	\la_2(\Om)= \inf_{\ga \in \Gamma}\sup_{u\in \ga}\left(\int_{\Om}|\na u|^{p} ~ dx +\int_{\mathbb{R}^N}\int_{\mathbb{R}^N} |u(x)-u(y)|^p J(x-y)~dxdy\right), 
	\end{align*}
	where $\Gamma =\{\ga \in C([-1,1], \mathcal{M}): \ga(-1)=-\phi_{1} \;\mbox{and}\; \ga(1)=\phi_1\}$, $\phi_1$ is the normalized eigenfunction corresponding to $\la_1(\Om)$ and $\mathcal{M}$ is defined \eqref{fs6}. 
\end{Theorem}
\begin{Theorem}\label{fsthm2}
	(Faber-Krahn inequality): Let $p\geq 2$, $c$ be a positive real number and $B$ be the ball of volume $c$. Then
	\begin{align*}
	\la_1(B) = \inf\left\{  \la_1(\Om), \; \Om \text{ open subset of } \mathbb{R}^N, \; |\Om|=c   \right\}.
	\end{align*}
\end{Theorem}
Next we will state theorem related to a sharp lower bound in $\la_2(\Om)$.
\begin{Theorem}\label{fsthm3}(Nonlocal Hong-Krahn-Szego inequality)
	Let $p\geq 2$ and $\Om \subset\mathbb{R}^N$ be an open bounded set. Assume $B$ is any ball of volume $|\Om|/2$. Then
	\begin{align}\label{fs15}
	\la_2(\Om)>\la_1(B). 
	\end{align}
	Moreover, equality is never attained in \eqref{fs15}, but the estimate is sharp in the following sense: if $\{s_n\}$ and $\{t_n\}$ are two sequences in $\mathbb{R}^N$ such that $\ds\lim_{n\ra \infty}|s_n-t_n|= +\infty$ and $\Om_n:= B_R(s_n)\cup B_R(t_n)$ then  $\ds \lim_{n\ra \infty}\la_2(\Om_n)= \la_1(B_R) $.
\end{Theorem}

The paper is organized as follows: In Section 2 we give the Variational Framework and Preliminary results. In Section 3 we  give the proof of Theorem \ref{fsthm1}.
In Section 4 we give the sharp lower bounds on the first and second eigenvalue of the operator associated to problem $(P_\la)$. In particular, we prove the Faber-Krahn inequality and nonlocal Hong-Krahn-Szego inequality. In Section 5, we discuss the eigenvalue problem associated with the combination of $p$-Laplacian  and fractional $p$-Laplacian. 
\section{Variational Framework and Preliminary results}

The energy functional $I: W^{1,p}_0 (\Om) \ra \mb R $ associated with problem $(P_\la)$ is   given by
\begin{align*} 
I(u)=  \int_{\Om}|\na u|^{p} ~ dx +\int_{\mathbb{R}^N}\int_{\mathbb{R}^N} |u(x)-u(y)|^p J(x-y)~dxdy -\la  \int_{\Om}|u|^p dx.
\end{align*}
 Note that $I$ is well defined on $W^{1,p}_0 (\Om)$ by extending  $u=0$ on $\mathbb{R}^N\setminus\Om$. Moreover, a direct computation show that  $I\in C^{1}( W^{1,p}_0 (\Om),\mb R)$ with
\begin{equation*}
\begin{aligned}
\langle I^{\prime}(u),\phi \rangle  =  p\; \mathcal{H}_{J,p}(u,\phi) 
 -  \la p \int_{\Om}|u|^{p-2}u \phi dx, 
\end{aligned}
\end{equation*}
for any $\phi\in  W^{1,p}_0 (\Om)$. 
\begin{Definition}
	A function $u \in W^{1,p}_0(\Om)$ is a solution of $(P_\la)$ if $u$ satisfies the  equation 
	\begin{align*}
	\mathcal{H}_{J,p}(u,\phi)= \la \int_{\Om}|u|^{p-2}u \phi~ dx,\; \; \text{for all } \phi \in W^{1,p}_0(\Om),
	\end{align*}
	\noi where
	\begin{align*}
	\mathcal{H}_{J,p}(u,\phi):=&  \int_{\Om}|\na u|^{p-2}\na u \cdot \na \phi~ dx\\ & \quad + \int_{\mathbb{R}^N} \int_{\mathbb{R}^N}  |u(x)-u(y)|^{p-2}(u(x)-u(y))(\phi(x)-\phi(y))J(x-y)~ dx dy.
	\end{align*}
\end{Definition}

 Also $\tilde{I}:= I|_{\mc M}$ is $C^1(W^{1,p}_0 (\Om),\mb R)$, where $\mc M$ is defined as
\begin{align}\label{fs6}
\mc M :=\left\{u\in W^{1,p}_0 (\Om):\;    \; S(u):= \int_{\Om}|u|^p=1\right\}.
\end{align}
 Hence, $u\in \mc M$ is a nontrivial weak solution of the problem $(P_\la)$.
 \begin{Proposition}\label{fsprop1}
 	\cite{AR} Let $Y$ be a Banach space, $F,G \in C^{1}(Y,\mb R)$, $M=\{u\in Y
 	\;|\; G(u)=1\}$ and $u$, $v\in M$. Let $\e>0$ such that
 	$\|u-v\|>\e$ and \[\inf\{F(w): w\in M \;\mbox{and}\;
 	\|w-u\|_{Y}=\e\}>\max\{F(u),F(v)\}.\]
 	\noi Assume that $F$ satisfies the Palais-Smale  condition on $M$ and that
 	\[\Gamma =\{\ga \in C([-1,1], M): \ga(-1)=u \;\mbox{and}\; \ga(1)=v\}\]
 	is non empty. Then $\ds c=\inf_{\ga \in \Gamma}\max_{u\in\ga[-1,1]}
 	F(u) >\max\{F(u),F(v)\}$ is a critical value of $F|_M$.
 \end{Proposition}
   Observe that 
 \begin{align*}
\tilde{I}(u)= \int_{\Om}|\na u|^{p} ~ dx +\int_{\mathbb{R}^N}\int_{\mathbb{R}^N} |u(x)-u(y)|^p J(x-y)~dxdy \geq \la_1(\Om) \int_{\Om}|u|^p,
 \end{align*}
  for all $u \in W^{1,p}_0 (\Om)$. It implies for any $u \in \mathcal{M}$, we have $\tilde{I}(u)\geq \la_1(\Om)$. Since  $\tilde{I}(\pm \phi_1)= \la_1(\Om)$, we deduce that $\pm \phi_1$ are the  two global minimum  of $\tilde{I}$ as well as critical points of $\tilde{I}$. 

\noi We will now find the third critical point via Proposition \ref{fsprop1}. 
A norm of derivative of the restriction $\tilde{I}$ of $I$
at $u\in \mc M$ is defined as
\[\|\tilde{I}^{\prime}(u)\|_{*}=\inf\{\|I^{\prime}(u)- t S^{\prime}(u)\|_{*}:  t\in \mb R\}.\]

\begin{Lemma}\label{fslem8}
$\tilde{I}$ satisfies the Palais-Smale  condition on $\mc M$.
\end{Lemma}

\begin{proof}
Let $\{u_n\}_{n\in N}$ be a sequence in $\mc M$ such that $\tilde{I}(u_n)\ra c$ and $\|\tilde{I}^{\prime}(u_n)\|_{*} \ra 0 $ for some $c \in \mathbb{R}$. As a consequence,  there exists sequence  $t_n\in \mb R$ such that  for all $\phi \in W^{1,p}_0 (\Om) $ and  for some $C>0$,
\begin{align}\label{fs2}
|I(u_n)|\leq C \text{ and } \left| \mathcal{H}_{J,p}(u_n, \phi) - t_n \int_{\Om}
|u_{n}|^{p-2} u_{n} \phi ~dx \right|\leq \e_{n}\|\phi \|
\end{align}
where $\e_n\ra 0$. From \eqref{fs2} and  Sobolev embedding, we obtain  $\{u_n\}$ is bounded in $W^{1,p}_0 (\Om)$. It implies up to a subsequence, still denoted by $u_n$, there exists 
$u \in W^{1,p}_0 (\Om)$ such that
$u_n\rightharpoonup u$ weakly in $W^{1,p}_0 (\Om)$. Moreover, $u_{n}\ra u $
strongly in $L^{p}(\Om)$ for all $1\leq p< p^*$ and $u_n \ra u $ a.e in $\Om$. Let  $\phi=u_n$
in \eqref{fs2}, we get
\[|t_k|\leq  \int_{\Om}| \na u_n|^{p} ~ dx + \int_{\mathbb{R}^N}\int_{\mathbb{R}^N}|u_n(x)-u_n(y)|^p J(x-y)~dxdy  + \e_{n}\|u_n\|\leq C.\]
Thus  $t_n$ is bounded sequence  i.e,  up to a subsequence  $t_n \ra t $  as $n \ra \infty$, for some $t \in \mathbb{R}$.\\
\textbf{Claim :} $u_n\ra u$
strongly in $W^{1,p}_0 (\Om)$. Since $u_n\rightharpoonup u$ weakly in $W^{1,p}_0 (\Om)$, we
get
\begin{equation}\label{fs3}
\begin{aligned}
\mathcal{H}_{J,p}(u,u_n) \ra \mathcal{H}_{J,p}(u,u)
\text{ as } n\rightarrow \infty.
\end{aligned}
\end{equation}
Using the inequality which states that: for all  $ a, b \in \mathbb{R}^{n}$, we have 
\begin{equation*}
\begin{aligned}
|a-b|^{r} \leq  \left\{
\begin{array}{ll}
C_{r}\left((|a|^{r-2}a-|b|^{r-2}b)(a-b)\right)^{\frac{r}{2}}\left(|a|^r+|b|^{r}\right)^{\frac{2-r}{2}} , &  \text{ if }1< r < 2, \\
2^{r-2}(|a|^{r-2}a-|b|^{r-2}b)(a-b)   & \text{ if } r \geq 2 .\\
\end{array} 
\right.
\end{aligned}	
\end{equation*}
with the fact that   $\ld \tilde{I^{\prime}}(u_n),(u_n-u)\rd= o(\e_n)$ and \eqref{fs3}, we deduce that 
\begin{align*}
\int_{\Om}|\na (u_n-u)|^p ~ dx + \int_{\mathbb{R}^N}\int_{\mathbb{R}^N} |(u_n-u)(x)-(u_n-u)(y)|^{p} J(x-y) ~ dx dy   \lra 0\;\mbox{as}\; n\ra\infty.
\end{align*}
Thus, $u_n$ converges strongly to  $u$ in $W^{1,p}_0 (\Om)$.\QED
\end{proof}
  Define 
\begin{align}\label{fs7}
\la _*=\ds \inf_{\ga \in \Gamma}\max_{u\in\ga[-1,1]} \tilde{I}(u), 
\end{align} 
 where $\Gamma =\{\ga \in C([-1,1], \mathcal{M}): \ga(-1)=-\phi_{1} \;\mbox{and}\; \ga(1)=\phi_1\}$. Let $\ga(t)= \frac{t\phi_1+(1-|t|)\phi}{\|t\phi_1+(1-|t|)\phi\|_{L^p}}$, where $\phi \not \in \mathbb{R} \phi_1$. It shows that $\Gamma$ is nonempty.  Using Proposition \ref{fsprop1},  $\la_*$ is a critical point of $\tilde{I}$ and $\la_*>\la_1(\Om)$.

\begin{Proposition}\label{fsprop2}
	Let $A$ and $B$ be two bounded open sets in $\mathbb{R}^N$ with $A \subsetneq \; B$ and $B$ connected  then $\la_1(A)> \la_1(B)$.
\end{Proposition}
\begin{proof}
	By definition of $\la_1(A)$,  $\la_1(A)\geq \la_1(B)$. Now, let if possible $\la_1(A)=\la_1(B)$ and let $\phi_A$ be normalized eigenfunction of $\la_1(A)$, it implies $\phi_A=0 $ on $\mathbb{R}^N\setminus A$. Therefore, 
	\begin{align*}
	 \int_{B}|\na u|^p ~dx &+ \int_{\mathbb{R}^N}\int_{\mathbb{R}^N}|\phi_A(x)-\phi_A(y)|^pJ(x-y)~dxdy\\
	& = \int_{A}|\na u|^p ~dx + \int_{\mathbb{R}^N}\int_{\mathbb{R}^N}|\phi_A(x)-\phi_A(y)|^pJ(x-y)~dxdy\\
	& = \la_1(A)\int_A |\phi_A|^p~dx\\
	& = \la_1(B)\int_B |\phi_A|^p~dx.
	\end{align*}
	This implies $\phi_A$ is an eigenfunction of $\la_B$. But this is impossible since $B$ is connected and $\phi_A$ vanishes on $B\setminus A \neq \emptyset$.\QED
\end{proof}

In \cite[Lemmas 3.5 and 3.6 ]{cfg} and  \cite[Lemma B.1]{second}  the following lemmas were proved:
\begin{Lemma}\label{fslem1} 
Let $\mc M= \{u\in W^{1,p}_0 (\Om) : \int_{\Om}|u|^p~dx =1\}$ then
 $\mc M$ is locally arcwise connected and any open connected subset $\mc S$ of $\mc M$ is arcwise connected. Moreover,
 If $\mc S^{'}$ is any connected component of an open set $\mc S\subset \mc M$, then $\partial \mc S^{\prime}\cap \mc S=\emptyset$.
\end{Lemma}
\begin{Lemma}\label{fslem2} 
Let $\mc S=\{u\in \mc M : \tilde{I}(u)<r\}$, then any connected component of $\mc S$ contains a critical point of $\tilde{I}$.
\end{Lemma}
\begin{Lemma}\label{fslem3} 
Let $1 \leq p \leq \infty$ and $U,V \in \mb R$ such that $U\cdot V \leq 0$. Define the following function
\[g(t)=|U -tV|^p+|U-V|^{p-2}(U-V)V|t|^p,\; t \in \mb R.\]
Then we have \[g(t)\leq g(1)=|U-V|^{p-2}(U-V)U,\; t \in \mb R.\]
\end{Lemma}
\begin{Lemma}\label{fslem4}
Let $\al\in (0,1)$ and $p>1$. For any non-negative functions $u$, $v \in W^{1,p}_0 (\Om)$, consider the function $\sigma_t (x):= \left[(1-t)v^p(x)+ tu^p(x)\right]^{\frac{1}{p}}$ for all $t\in[0, 1]$. Then for all $t \in [0,1]$, 
\begin{equation*}
\begin{aligned}
  \int_{\mathbb{R}^N}\int_{\mathbb{R}^N} |\sigma_t (x)-\sigma_t (y)|^p J(x-y) ~dxdy  & \leq (1-t) \int_{\mathbb{R}^N}\int_{\mathbb{R}^N} |v (x)-v (y)|^p J(x-y) ~dxdy \\& \quad + t\int_{\mathbb{R}^N}\int_{\mathbb{R}^N} |u (x)-u (y)|^p J(x-y) ~dxdy.
 \end{aligned}
\end{equation*}

\end{Lemma}
\begin{proof}
Proof is analogous to \cite[Lemma 4.1]{pal}. \QED
\end{proof}
\section{Proof of Theorem \ref{fsthm1}}

\begin{Lemma}\label{fslem5}
Let $1<p<\infty$. Then number $\la_*$ (defined  in \eqref{fs7}) is the second smallest eigenvalue  of $\mathcal{L}_{J,p}$.
\end{Lemma}
\begin{proof} On the contrary assume  that there exists an eigenvalue $s$ such that  $\la_{1}(\Om)<s <\la_*$. It implies that $s$ is a critical value of  $\tilde{I}$ . Since $\la_1(\Om)$ is isolated, we may assume that $\tilde{I}$ has no critical value in
$(\la_{1}(\Om),s)$. To get a contradiction, it is enough to construct a path $\ga$ connecting from $\phi_{1}$ to $-\phi_{1}$ such that $\tilde{I}(\ga)\leq s $.

\noi Let $u\in \mc M$ be a critical point of $\tilde{I}$ at level
$s$. Then $u$ satisfies,
\begin{equation}\label{fs8}
\begin{aligned}
\mathcal{H}_{J,p}(u,\phi)=\la_*\int_{\Om}|u|^{p-2}u \phi ~dx
 \text{ for all  } \phi \in {W}^{1,p}_0(\Om).
\end{aligned}
\end{equation}
\noi Since, $u$  changes sign in $\Om$ . Taking $\phi= u^{+}$ and $\phi= u^-$ in \eqref{fs8}, we get 
{\small\begin{equation}\label{fs9} \int_{\Om}|\na u^+ |^{p} ~dx+ \int_{\mathbb{R}^N}\int_{\mathbb{R}^N} |u(x)-u(y)|^{p-2}(u(x)-u(y))(u^+(x)-u^+(y)) J(x-y)~ dxdy = \la_* \int_{\Om}(u^{+})^p dx,
\end{equation}}
and 
{\small\begin{equation}\label{fs10}
\int_{\Om}|\na u^- |^{p} ~dx-  \int_{\mathbb{R}^N}\int_{\mathbb{R}^N} |u(x)-u(y)|^{p-2}(u(x)-u(y))(u^-(x)-u^-(y))J(x-y)~ dxdy = \la_*\int_{\Om} (u^{-})^p dx.
\end{equation}}
So as  a consequence, we have
\begin{align}\label{fs11}
& \int_{\Om}|\na u^+ |^{p} ~dx+ \int_{\mathbb{R}^N}\int_{\mathbb{R}^N} |u^+(x)-u^{+}(y)|^pJ(x-y)~ dxdy \leq \la_* \int_{\Om}|u^+|^p~dx,\\ \label{fs12}
& \int_{\Om}|\na u^- |^{p} ~dx+ \int_{\mathbb{R}^N}\int_{\mathbb{R}^N} |u^-(x)-u^{-}(y)|^pJ(x-y)~ dxdy \leq \la_* \int_{\Om}|u^-|^p~dx.
\end{align}
It further implies that 
\[\tilde{I}(u)=\la_*,\; \tilde{I}\left(\frac{ u^+}{ \|u^+\|_{L^p}}\right)
\leq \la_*,\tilde{I}\left(\frac{u^-}{\|u^-\|_{L^p}}\right)\leq \la_*,\tilde{I}\left( \frac{-u^-}{\|u^-\|_{L^p}}\right)\leq \la_* .\] Now, we will define three paths in $\mc M$ which go $u$ to  $\frac{ u^+}{ \|u^+\|_{L^p}}$ , $\frac{ u^+}{ \|u^+\|_{L^p}}$ to $\frac{u^-}{\|u^-\|_{L^p}}$ and $\frac{u^-}{\|u^-\|_{L^p}}$ to $\frac{-u^-}{\|u^-\|_{L^p}}:$
\begin{align*}
& \ga_{1}(t)=\frac{u^+- (1-t)u^-}{ \|u^+- (1-t)u^-\|_{L^p}}, \;  \ga_2(t)=\frac{[(1-t)(u^{+})^p+ t(u^{-})^p]^{1/p}}{\|(1-t)(u^{+})^p+ t(u^{-})^p\|_{L^p}} ,\;  \ga_3(t)=\frac{(1-t)u^{+}-u^{-}}{ \|(1-t)u^{+}-u^{-}\|_{L^p}}.
\end{align*}
Taking into account \eqref{fs9}, \eqref{fs10} and Lemma \ref{fslem3} with $U=u^+(x)-u^+(y)$ and $V=u^-(x)-u^-(y)$, we deduce that for all $t\in[0,1]$,
\begin{align*}
\tilde{I}(\ga_1(t)) &\leq \frac{ \ds \int_{\Om}|\na u^+|^p~dx + \displaystyle\int_{\mathbb{R}^N}\int_{\mathbb{R}^N} |U-V|^{p-2}(U-V)U J(x-y)~dxdy }{ \|u^+- (1-t)u^-\|_{L^p}^p}\\
&\quad +\frac{|1-t|^p\left[\ds \int_{\Om}|\na u^-|^p~dx -\displaystyle\int_{\mathbb{R}^N}\int_{\mathbb{R}^N}|U-V|^{p-2}(U-V)V J(x-y)~dxdy \right]}{\|u^+- (1-t)u^-\|_{L^p}^p}\\
&=\la_*.
\end{align*}
By means of Lemma \ref{fslem4},  we deduce
\begin{equation*}
\begin{split}
\tilde{I}(\ga_2(t)) &\leq \frac{(1-t)\left[\ds \int_{\Om}|\na u^+|^p~dx +\displaystyle\int_{\mathbb{R}^N}\int_{\mathbb{R}^N}|u^+(x)-u^+(y)|^p J(x-y)~dxdy \right] }{ \|(1-t)(u^{+})^p+ t(u^{-})^p\|^p_{L^p}}\\
& \quad +\frac{t \left[ \ds \int_{\Om}|\na u^-|^p~dx+\displaystyle\int_{\mathbb{R}^N}\int_{\mathbb{R}^N}|u^-(x)-u^-(y)|^p J(x-y)~dxdy \right]}{ \|(1-t)(u^{+})^p+ t(u^{-})^p\|^p_{L^p}}\\
& \leq \la_*.
\end{split}
\end{equation*}
Once again from \eqref{fs9}, \eqref{fs10} and Lemma \ref{fslem3}  with  $U=u^-(y)-u^-(x)$ and $V=u^+(y)-u^+(x)$, we obtain
\begin{align*}
\tilde{I}(\ga_3(t)) &\leq \frac{  \ds \int_{\Om}|\na u^-|^p~dx+ \displaystyle\int_{\mathbb{R}^N}\int_{\mathbb{R}^N} |U-V|^{p-2}(U-V)U J(x-y) ~dxdy  }{  \|(1-t)u^{+}-u^{-}\|_{L^p}^p}\\&\quad + \frac{ |1-t|^p \left[\ds \int_{\Om}|\na u^+|^p~dx  - \displaystyle \int_{\mathbb{R}^N}\int_{\mathbb{R}^N} |U-V|^{p-2}(U-V)V J(x-y)~dxdy \right]}{\|(1-t)u^{+}-u^{-}\|_{L^p}^p}\\
&=\la_*.
\end{align*}
 Clearly $\pm \phi_{1}\in \mc S$, where  $\mc S = \{v\in\mc M : \tilde{I}(v)<\la_* \}$.  Also, $\frac{ u^-}{\|u^-\|_{L^p}} $ is not a critical point of $\tilde{I}$, thanks to the fact that  $\frac{ u^-}{\|u^-\|_{L^p}} $ does not change sign and vanishes on a
set of positive measure. Therefore, there exists a $C^1$ path $\sigma : [-\de,\de]\ra
\mc M$ with $\sigma(0)= \frac{u^-}{\|u^-\|_{L^p}}$ and
$\frac{d}{dt}\tilde{I}(\sigma(t))|_{t=0}\ne 0$. With the help of  this path we
can move from $\frac{ u^-}{\|u^-\|_{L^p}}$ to a point $v$ with
$\tilde{I}(v)<\la_*$. Consider a  connected component of $\mc S$
containing $v$ and employing Lemma \ref{fslem2} we get 
$\phi_{1}$ (or $-\phi_{1}) $ is in this component. Let us assume that it is
$\phi_{1}$. At this point  we construct a path $\ga_{4}(t)$ from
$\frac{ u^-}{\|u^-\|_{L^p}}$ to $\phi_{1}$ which is at level
less than $\la_*$. Consider the symmetric path $-\ga_{4}(t)$ connects
$\frac{- u^-}{\|u^-\|_{L^p}}$ to $-\phi_{1}$. Since $\tilde{I}$ is even,
\[\tilde{I}(-\ga_4(t))= \tilde{I}(\ga_4(t))\leq \la_* \;\mbox{for all}\; t .\]
Lastly, we can connect  $\ga_1(t)$, $\ga_2(t)$ and $\ga_4(t)$, to obtain a path from $u$
to $\phi_{1}$ and joining $\ga_3(t)$ and $-\ga_4(t)$ we get a path from
$u$ to $-\phi_{1}$. Taking account all this together, we get a path  in $\mc M$ 
from $\phi_{1}$ to $-\phi_{1}$  at levels  $\leq \la_*$
for all $t$. This completes the proof.\QED
\end{proof}
\textbf{Proof of Theorem \ref{fsthm1} :}  By Theorem 3.3 of \cite{rossi1}, there exists a positive number  $\la_2(\Om)$ given by 
\begin{align*}
\la_2(\Om) = \inf_{A\in \mathcal{A}}\sup_{ u\in A} \mathcal{H}_{J,p}(u,u), 
\end{align*}
where $\mathcal{A}= \{ A \subset \mathcal{M}: A \text{ compact, symmetric, of genus } \geq 2  \}$.    Let $\ga$ be a curve in $\La$ then by joining this with its symmetric path $-\ga$ we obtain a set of genus $\geq 2$ where $\tilde{I}$ does not increase its value. Hence, $ \la_2(\Om)\leq \la_*$ (defined  in \eqref{fs7}). From Lemma \ref{fslem5},  $\la_*$ is the smallest eigenvalue. That is, there is no eigenvalue between $\la_1(\Om)$ and $\la_*$, it implies $\la_*\leq \la_2(\Om)$. Therefore, $\la_2(\Om)$ is second eigenvalue of  the operator $\mathcal{L}_{J,p}$ with variational characterization 
\begin{align*}
\la_{2}(\Om) := & \inf_{\ga \in \Gamma}\sup_{u\in \ga}\left(\int_{\Om}|\na u|^p~dx \int_{\mathbb{R}^N}\int_{\mathbb{R}^N} |u(x)-u(y)|^p J(x-y)~dxdy \right),
\end{align*}
 where $\Gamma =\{\ga \in C([-1,1], \mathcal{M}): \ga(-1)=-\phi_{1} \;\mbox{and}\; \ga(1)=\phi_1\}$.



\section{ Proof of Theorems \ref{fsthm2} and \ref{fsthm3}}
In this Section we will give a sharp lower bound on $\la_1(\Om)$ and $\la_{2}(\Om)$ in terms of volume of $\Om$. We will assume that $p\geq 2$ and  $J$ is radially symmetric decreasing nonnegative continuous function with compact support, $J(0)>0$ and $\int_{\mathbb{R}^N} J(x)~ dx =1$. With this assumption, $J^*(x)= J(x)$, where $J^*$ stands for the  symmetric decreasing rearrangement of the function $J$. Also, we have the following  Polya-Szego  inequality:  
	\begin{align}\label{fs13}
	 \int_{\mathbb{R}^N}\int_{\mathbb{R}^N} |u^*(x)-u^*(y)|^{p}J(x-y)~ dxdy \leq  \int_{\mathbb{R}^N}\int_{\mathbb{R}^N} |u(x)-u(y)|^{p}J(x-y)~ dxdy. 
	\end{align}
	For the proof of \eqref{fs13}, we refer \cite[Corrollary 2.3]{leib1}. 	\\
	
\textbf{Proof of Theorem \ref{fsthm2} :}
	Let $\Om$  be a bounded open set of volume  $c$  and $\Om^* = B$ the ball of same volume. Let $\phi_1$ be the eigenfunction corresponding to $\la_1(\Om)$ and $\phi_1^*$ be the Schwarz symmetrization of the function $\phi_1$ then by  Polya-Szego  inequality (See \cite[Theorem 2.1.3]{henrot} and \cite[Corrollary 2.3]{leib1}), we have 
\begin{equation}
	\begin{aligned}\label{fs20}
& \int_{\Om^*}|\na \phi_1^*|^p~dx + \int_{\mathbb{R}^N}\int_{\mathbb{R}^N} |\phi_1^*(x)-\phi_1^*(y)|^{p}J(x-y)~ dxdy \\
&  \quad  \leq   \int_{\Om}|\na \phi_1|^p~dx +\int_{\mathbb{R}^N}\int_{\mathbb{R}^N} |\phi_1(x)-\phi_1(y)|^{p}J(x-y)~ dxdy.
	\end{aligned}
	\end{equation}
	Moreover, we know that $\ds  \int_{\Om^*}| \phi_1^*|^p~dx=  \int_{\Om}| \phi_1|^p~dx$. Therefore by definition of $\la_1(\Om)$, we obtain
	\begin{align*}
\la_1(\Om^*) & \leq  \frac{ \ds \int_{\Om^*}|\na \phi_1^*|^p~dx + \int_{\mathbb{R}^N}\int_{\mathbb{R}^N} |\phi_1^*(x)-\phi_1^*(y)|^{p}J(x-y)~ dxdy}{\|\phi_1^*\|^p_{L^p}}\\
& \leq \frac{ \ds \int_{\Om}|\na \phi_1|^p~dx + \int_{\mathbb{R}^N}\int_{\mathbb{R}^N} |\phi_1(x)-\phi_1(y)|^{p}J(x-y)~ dxdy}{\|\phi_1\|^p_{L^p}} = \la_1(\Om).
	\end{align*}
	 Furthermore, if $\la_1(\Om)= \la_1(B)$ then equality must hold in \eqref{fs20}.  Then using \cite[Lemma A.2]{frank}, we have that $\phi$ is a translation of a radially symmetric decreasing 	function. It implies that $\Om$ is a ball. 	It yields the required result. \QED

\begin{Lemma}\label{fslem6}
	Let $1<p<\infty$ and $a,\; b \in \mathbb{R}$ then the following holds:
\begin{enumerate}
	\item[(i)] There exists $c_p>0$ such that 
	\begin{align*}
	|a-b|^p\leq |a|^p+|b|^p+c_p(|a|^2+|b|^2)^{\frac{p-2}{2}}|ab|
	\end{align*}
	\item [(ii)] 	If $ab\leq 0$ then 
	\begin{equation*}
|a-b|^{p-2}(a-b)a \geq \;  \left\{
	\begin{array}{ll}
	|a|^p-(p-1)|a-b|^{p-2}ba  &  \text{ if }1<p<2, \\
	|a|^p-(p-1)|a|^{p-2}ba & \text{ if } p>2. \\
	\end{array} 
	\right. 
	\end{equation*}
\end{enumerate}
\end{Lemma}
\begin{proof}
For detailed proof,  see 	\cite[Lemmas B.2 and B.3]{second}.\QED
\end{proof}
\begin{Lemma}\label{fslem7}
	(Nodal domains) Let $\la>\la_1(\Om)$ be an eigenvalue of $\mathcal{L}_{J,p}$ and $\phi_\la$ be the associated eigenfunction. Assume the set 
	\begin{align*}
	\Om^+:= \{ x \in \Om : \phi_\la(x)>0 \} \quad \text{and} \quad 
		\Om^-:= \{ x \in \Om : \phi_\la(x)<0 \} .
\end{align*}
Then $\la> \{ \la_1(\Om^+),\; \la_1(\Om^-) \}$. 
\end{Lemma}
\begin{proof}
	By \cite[Corrollary 3.1]{rossi1}, we have $\phi_\la \in C^{1,\al}(\overline{\Om})$ for some $\al \in (0,1)$. Therefore, $\Om^+$ and $\Om^-$ are open subsets of $\Om$ and hence $\la_1(\Om^+) $ and $\la_1(\Om^-)$ are well defined. Also, from  \cite[Lemma 3.3]{rossi1} $\phi_\la$ changes sign in $\Om$.  Since $\phi_\la$ is an eigenfunction, it implies 
	\begin{align}\label{fs14}
	\mathcal{H}_{J,p}(\phi_\la,v)= \la \int_{\Om}|\phi_\la|^{p-2}\phi_\la v ~ dx,\; \; \text{for all } v \in W^{1,p}_0(\Om). 
	\end{align}
		Let $v= \phi_\la^+$. Using Lemma \ref{fslem2}(ii) with  $ a= \phi_\la^+(x)-\phi_\la^+(y)$ and $b= \phi_\la^-(x)-\phi_\la^-(y)$  then we have 
	\begin{align*}
	& \int_{\Om^+}|\na \phi_\la^+|^p~dx + \int_{\mathbb{R}^N}\int_{\mathbb{R}^N} |\phi_\la^+(x)-\phi_\la^+(y)|^{p}J(x-y)~ dxdy \\
	&   <  \int_{\Om^+}|\na \phi_\la^+|^p~dx +\int_{\mathbb{R}^N}\int_{\mathbb{R}^N} |\phi_\la(x)-\phi_\la(y)|^{p-2}(\phi_\la(x)-\phi_\la(y))(\phi_\la^+(x)-\phi_\la^+(y))J(x-y)~ dxdy\\
	& = \la \int_{\Om^+}|\phi_\la^+|^{p}~dx .
	\end{align*}
	Taking in to account that $\phi_\la^+$ is admissible in variational framework defined for $\la_1(\Om^+)$. Indeed, 
	\begin{align*}
	\la_1(\Om^+) \int_{\Om^+}|\phi_\la^+|^{p}~dx\leq \int_{\Om^+}|\na \phi_\la^+|^p~dx + \int_{\mathbb{R}^N}\int_{\mathbb{R}^N} |\phi_\la^+(x)-\phi_\la^+(y)|^{p}J(x-y)~ dxdy.
	\end{align*}
	Therefore, $\la> \la_1(\Om^+)$. Now for the set $\Om^-$, we will proceed analogously as above with $v= \phi_\la^-, \; a=\phi_\la^-(x)-\phi_\la^-(y) $ and $b=  \phi_\la^+(x)-\phi_\la^+(y)$ to achieve $\la> \la_1(\Om^-)$. Hence we get the desired result. \QED
	\end{proof}
\textbf{Proof of Theorem \ref{fsthm2} :}
	Let $\phi_2$ be the eigenfunction corresponding to  the eigenvalue $\la_2(\Om)$, let 
	\begin{align*}
	\Om^+:= \{ x \in \Om : \phi_2(x)>0 \} \quad \text{and} \quad 
	\Om^-:= \{ x \in \Om : \phi_2(x)<0 \} .
	\end{align*}
	It implies $|\Om^+|+ |\Om^-|\leq |\Om|$ and using Lemma \ref{fslem7} and Theorem \ref{fsthm2}, we have 
	\begin{align*}
	\la_2(\Om)> \la_1(\Om^+)\geq \la_1(B_{r_1}) \quad \text{ and } \quad \la_2(\Om)> \la_1(\Om^+)\geq \la_1(B_{r_2}),
	\end{align*}
	where $B_{r_1}$ and $B_{r_2}$ are two balls such that $|B_{r_1}| =|\Om^+|$ and $|B_{r_2}| =|\Om^-|$. Hence 
	\begin{align*}
	\la_2(\Om)> \max \{  \la_1(B_{r_1}),\;  \la_1(B_{r_2}) \} \quad \text{and} \quad  |B_{r_1}|+ |B_{r_2}|\leq |\Om|.
	\end{align*}
	\textbf{Claim: } $\max \{  \la_1(B_{r_1}),\;  \la_1(B_{r_2}) \}$ is minimized when $|B_{r_1}|= |B_{r_2}|= |\Om|/2$.\\
Let $B_r$ be a ball such that $|B_r|=|\Om|/2 $.	Since $|B_{r_1}|+ |B_{r_2}|\leq |\Om|$ therefore we will divide the proof of claim in three cases.\\
	\textbf{Case 1: } If $|B_{r_1}|, |B_{r_2}|\leq  |\Om|/2$.\\
	It implies that balls $B_{r_1},\;B_{r_2}$ are contained in ball $B_r$ then by Proposition \ref{fsprop2} we have 
	$\la_1(B_r)\leq \la_1(B_{r_1}),\;  \la_1(B_{r_2})$. It implies $\max \{  \la_1(B_{r_1}),\;  \la_1(B_{r_2}) \} \geq  \la(B_r)$. 
	\\
	\textbf{Case 2: } If $  |B_{r_1}|< |\Om|/2< |B_{r_2}| $.\\
It implies $|B_{r_1}|< |B_r|< |B_{r_2}|$. From Proposition \ref{fsprop2}, we have $\la_1(B_{r_1})\leq \la_1(B_r)\leq   \la_1(B_{r_2})$. Thus,  $\max \{  \la_1(B_{r_1}),\;  \la_1(B_{r_2}) \}\geq \la_1(B_{r_2})\geq  \la_1(B_r)$.\\
\textbf{Case 3: } If $  |B_{r_2}|< |\Om|/2< |B_{r_1}| $.\\
Similarly as in case 2 we have  $\max \{  \la_1(B_{r_1}),\;  \la_1(B_{r_2}) \}\geq  \la_1(B_r)$.\\
  Hence,  from all cases we have $\max \{  \la_1(B_{r_1}),\;  \la_1(B_{r_2}) \}$ is minimized only when $|B_{r_1}|= |B_{r_2}|= |\Om|/2$. It proves \eqref{fs15}.\\
   Now for equality we define $\Om_n:= B_r(s_n)\cup B_r(t_n)$, where  
  $\{s_n\}$ and $\{t_n\}$ are  sequences in $\mathbb{R}^N$ such that $|s_n-t_n|$ diverges as $n \ra \infty$. Let $\phi_{s_n}$ and $\phi_{t_n}$ are the positive normalized eigenfunctions on $B_R(s_n)$ and $B_R(t_n)$ respectively.  Let $ f : \mathbb{S}^1 \ra \mathcal{M}$ given by 
  \begin{align*}
  f(\theta_1, \theta_2)= \frac{|\theta_1|^{\frac{2-p}{p}}\theta_1\phi_{s_n}- |\theta_2|^{\frac{2-p}{p}}\theta_2\phi_{t_n}}{ \bigg \| |\theta_1|^{\frac{2-p}{p}}\theta_1\phi_{s_n}- |\theta_2|^{\frac{2-p}{p}}\theta_2\phi_{t_n}\bigg \|_{L^p}}
  \end{align*}
  Then define $A= \text{Range}(f)$. It implies that $A$ is compact, symmetric, and of genus $\geq 2$. Now  taking in account the definition of $\la_2(\Om)$ and Lemma \ref{fslem6}(ii) with 	 $a= \phi_{s_n}(x)-\phi_{s_n}(y)$ and $b= \phi_{t_n}(x)-\phi_{t_n}(y)$, we obtain 
  \begin{align*}
 \la_2(\Om_n)& \leq \max_{|\theta_1|^p+|\theta_2|^p=1} \bigg \{ \int_{\Om_n}|\na (\theta_1\phi_{s_n}-\theta_2\phi_{t_n})|^p~dx+ \int_{\mathbb{R}^N}\int_{\mathbb{R}^N}|\theta_1a- \theta_2b|^pJ(x-y)~dxdy\bigg\}\\
 & = \max_{|\theta_1|^p+|\theta_2|^p=1} \bigg \{ \int_{\Om_n}|\na \theta_1\phi_{s_n}|^p~dx+ \int_{\Om_n}|\na \theta_2\phi_{t_n}|^p~dx\\
 & \hspace{4 cm} + \int_{\mathbb{R}^N}\int_{\mathbb{R}^N}|\theta_1a- \theta_2b|^pJ(x-y)~dxdy\bigg\}\\
 & \leq  \max_{|\theta_1|^p+|\theta_2|^p=1} \bigg \{ \int_{\Om_n}|\na \theta_1\phi_{s_n}|^p~dx+ \int_{\Om_n}|\na \theta_2\phi_{t_n}|^p ~dx\\
 &  \hspace{2.5cm}+ \int_{\mathbb{R}^N}\int_{\mathbb{R}^N}|\theta_1a|^pJ(x-y)~dxdy + \int_{\mathbb{R}^N}\int_{\mathbb{R}^N}| \theta_2b|^pJ(x-y)~dxdy\\
 & \hspace{3cm} + c_p \int_{\mathbb{R}^N}\int_{\mathbb{R}^N}(|\theta_1a|^2+ |\theta_2b|^2)^{\frac{p-2}{2}} |\theta_1\theta_2ab|J(x-y) ~dxdy
 \bigg\}\\
 & = \la_1(B_R) + c_p\max_{|\theta_1|^p+|\theta_2|^p=1} \int_{\mathbb{R}^N}\int_{\mathbb{R}^N}(|\theta_1a|^2+ |\theta_2b|^2)^{\frac{p-2}{2}} |\theta_1\theta_2ab|J(x-y) ~dxdy. 
  \end{align*}
	Since $ab= -\phi_{s_n}(x)\phi_{t_n}(y)- \phi_{s_n}(y)\phi_{t_n}(x)$ is nonzero only when $(x,y)\in B_R(s_n)\times B_R(t_n) \cup  B_R(t_n)\times B_R(s_n)$. And $s_n-t_n-2R<x-y$ for all  $(x,y)\in B_R(s_n)\times B_R(t_n) \cup  B_R(t_n)\times B_R(s_n)$. Hence 
	\begin{align*}
	 \la_2(\Om_n)& \leq \la_1(B_R)\\
	 &  +  2J(s_n-t_n-2R)c_p\max_{|\theta_1|^p+|\theta_2|^p=1} \int_{B_R(s_n)}\int_{B_R(t_n)}(|\theta_1a|^2+ |\theta_2b|^2)^{\frac{p-2}{2}} |\theta_1\theta_2ab| ~dxdy.
	\end{align*}
	Since   \begin{align*}
	 2c_p\max_{|\theta_1|^p+|\theta_2|^p=1} \int_{B_R(s_n)}\int_{B_R(t_n)}(|\theta_1a|^2+ |\theta_2b|^2)^{\frac{p-2}{2}} |\theta_1\theta_2ab| ~dxdy< \infty
	\end{align*} and $J(s_n-t_n-2R) \ra 0 $ as $n \ra \infty$. Thus $ \ds \lim_{n\ra \infty}\la_2(\Om_n) \leq \la_1(B_R)$. This proved the desired result.  \QED

  \section{Remarks on the eigenvalues of combination of $p$-Laplacian and fractional $p$-Laplacian}
We  consider the  following eigenvalue problem:   
 \begin{equation*}
 (F_\la)\;
 \left.\begin{array}{rllll}
 \mathcal{L}(u) =\la |u|^{p-2}u \text{ in } \Om,\; u=0 \text{ in } \mathbb{R}^N\setminus \Om,
 \end{array}
 \right.
 \end{equation*}
 where  $1< p< \infty$  and 
 the operator $\mathcal{L}(u)$ is defined as 
 $\mathcal{L}(u):= -\De_p u +(-\De)^s_p u $ where $\De_p u$ is the usual $p$-Laplacian operator and  $(-\De)^s_p u$ is the fractional  $p$-Laplacian is given by 
 \begin{align*}
(-\De)^s_p u(x):= 2 \text{ P.V } \ds  \int_{\mathbb{R}^N} \frac{|u(x)-u(y)|^{p-2}(u(x)-u(y))}{|x-y|^{N+ps}}~ dy,
\end{align*} 
 where   $\Om \subset\mathbb{R}^N (N>ps)$ be a bounded open set, $0<s<1$. 
\\
 \begin{Definition}
 	A function $u \in W^{1,p}_0(\Om)$ is a solution of $(F_\la)$ if $u$ satisfies the  equation 
 	\begin{align*}
 	\mathcal{H}(u,\phi)= \la \int_{\Om}|u|^{p-2}u \phi~ dx,\; \; \text{for all } \phi \in W^{1,p}_0(\Om),
 	\end{align*}
 	\noi where
 	\begin{align*}
 	\mathcal{H}(u,\phi):=&  \int_{\Om}|\na u|^{p-2}\na u \cdot \na \phi~ dx\\ & \quad + \int_{\mathbb{R}^N} \int_{\mathbb{R}^N}  
 	\frac{|u(x)-u(y)|^{p-2}(u(x)-u(y))(\phi(x)-\phi(y))}{|x-y|^{N+ps}}~ dx dy
 	\end{align*}
 \end{Definition}
 The energy functional associated with problem $(F_\la)$ is the functional  $\mathcal{I}: W^{1,p}_0 (\Om) \ra \mb R $ given by
 \begin{align*}
 \mathcal{I}(u)=  \int_{\Om}|\na u|^{p} ~ dx +\int_{\mathbb{R}^N} \int_{\mathbb{R}^N}  
 \frac{|u(x)-u(y)|^{p}}{|x-y|^{N+ps}}~ dx dy-\la \int_{\Om}| u|^{p} ~ dx .
 \end{align*}
 Let $u \in C_c^{\infty}(\Om)$ then by extending $u=0$ on $\mathbb{R}^N\setminus \Om$, we see that 
 \begin{align*}
\int_{\mathbb{R}^N} \int_{\mathbb{R}^N}  
\frac{|u(x)-u(y)|^{p}}{|x-y|^{N+ps}}~ dx dy= \int_{Q}  
\frac{|u(x)-u(y)|^{p}}{|x-y|^{N+ps}}~ dx dy, \; \text{ where }Q= \mathbb{R}^{2N}\setminus (\Om^c\times \Om^c). 
 \end{align*} 
 Also, it is not difficult to show that  \[\int_{Q}  
 \frac{|u(x)-u(y)|^{p}}{|x-y|^{N+ps}}~ dx dy\le C \|\na u\|_{L^p}^{p}\;  \text{for all}\; u \in C_c^{\infty}(\Om).\] By density, we get $\mathcal{I}$ is well defined on $W^{1,p}_{0}(\Om).$ Also,  $\mathcal{I}\in C^{1}( W^{1,p}_0 (\Om),\mb R)$. Moreover, $\tilde{ \mathcal{I}}:=  \mathcal{I}|_{\mc M}$ is $C^1(W^{1,p}_0 (\Om),\mb R)$, where $\mc M$ is defined as in \eqref{fs6}. 
 By using the same assertions and arguments as in the proofs of Theorem \ref{fsthm4} and Theorem \ref{fsthm1} we can obtain Theorems \ref{fsthm4} and \ref{fsthm1} for the operator $\mathcal{L}$.


\begin{thebibliography}{11}
	\bibitem{AR} A.  Ambrosetti  and P. H. Robinowitz, {\it Dual variational methods in critical point theory and applications}, Journal of Functional Analysis 14 (1973), 349-381.
	
	\bibitem{leib1}F.J Almgren and E.H Lieb, {\it  Symmetric decreasing rearrangement is sometimes continuous}, Journal of the American Mathematical Society (1989), 683-773.
	
	\bibitem{an1} F. Andreu-Vaillo, J. M. Mazon, J. D. Rossi and J. J. Toledo-Melero, {\it  Nonlocal Diffusion
	Problems}, American Mathematical Society, Mathematical Surveys and Monographs  165 (2010).


\bibitem{an2}F. Andreu, J. M. Mazon, J. D. Rossi and J. Toledo, {\it  The limit as $p \ra \infty$ in a nonlocal
p−Laplacian evolution equation.  A nonlocal approximation of a model for sandpiles}. Calculus of Variation and  Partial Differential  Equations 35 (2009), no. 3,  279-316.

\bibitem{an3} F. Andreu, J. M. Mazon, J. D. Rossi and J. Toledo, {\it  A nonlocal $p$-Laplacian evolution
equation with non homogeneous Dirichlet boundary conditions}, SIAM Journal of  Mathematical  Analysis 40 (2009), no. 5, 1815-1851.

	
	\bibitem{bisci} G.M Bisci, V. D. Radulescu and R. Servadei, {\it Variational methods for nonlocal fractional problems},   Cambridge University Press 162 (2016).
	
	
	\bibitem{second} L. Brasco and E. Parini, {\it The second eigenvalue of the fractional $p$-Laplacian}. Advances in Calculus of Variations 9  (2016), no. 4,  323-355.
	
	\bibitem{cfg} M. Cuesta, D. de Figueiredo and  J. P. Gossez, {\it The Beginning of the Fu\v{c}ik Spectrum for the $p$-Laplacian}, Journal of Differential
	Equations 159 (1999),  212-238.

\bibitem{frank}R.L. Frank and R. Seiringer, { \it Non-linear ground state representations and sharp Hardy inequalities}, Journal of Functional Analysis 255 (2008), no. 12, 3407-3430.
	
	\bibitem{pal} G. Franzina and G. Palatucci,  {\it Fractional $p$-eigenvalues}, Rivista di Matematica della Universit\`a di Parma 5 (2014), no. 2,  373-386.
	
	\bibitem{goel1}D. Goel, S. Goyal and  K. Sreenadh, {\it 	First curve of Fu\v{c}ik spectrum for the p-fractional Laplacian operator with nonlocal normal boundary conditions},
	Electronic Journal of Differential Equations (2018), no. 74,  1-21.
	
	
	\bibitem{hardy} S. Goyal, {\it  On the eigenvalues and Fu\v{c}ik spectrum of $p$-fractional Hardy-Sobolev operator with weight function}, Applicable Analysis (2017),  1-26.
	
	\bibitem{henrot}A. Henrot, {\it  Extremum problems for eigenvalues of elliptic operators}, Springer Science and Business Media  (2006).
	
	
	
\bibitem{hitchhker}	E. D. Nezza, G. Palatucci and E. Valdinoci, {\it Hitchhikerʼs guide to the fractional Sobolev spaces}, Bulletin des Sciences Math\`ematiques 136 (2012), 521-573.
	
	

	
	\bibitem{rossi1} L.M Del Pezzo, R. Ferreira and J.D Rossi, {\it Eigenvalues for a combination between local and nonlocal $ p$-Laplacians}, arXiv:1803.07988.
\bibitem{sree1}	K. Sreenadh, {\it On the second eigenvalue of a Hardy-Sobolev operator},  Electronic Journal of Differential Equations (2004). 
	
	
	
\end{thebibliography}
\end{document}